\pdfoutput=1
\RequirePackage{ifpdf}
\ifpdf % We are running pdfTeX in pdf mode
\documentclass[pdftex]{sigma}
\else
\documentclass{sigma}
\fi

\numberwithin{equation}{section}

\newtheorem{Theorem}{Theorem}[section]
\newtheorem{Corollary}[Theorem]{Corollary}
\newtheorem{Lemma}[Theorem]{Lemma}
\newtheorem{Proposition}[Theorem]{Proposition}
 { \theoremstyle{definition}
\newtheorem{Definition}[Theorem]{Definition}
\newtheorem{Example}[Theorem]{Example}
\newtheorem{Remark}[Theorem]{Remark}
\newtheorem{Remarks}[Theorem]{Remarks} }

\def\W{\mathcal{W}}
\def\M{\mathcal{M}}
\def\J{\mathcal{J}}

\def\V{\mathcal{V}}
\def\S{\mathcal{S}}
\def\C{\mathcal{C}}
\def\Ham{\mathcal{H}}

\def\R{\mathbb{R}}

\def\L{\mbox{Leg}}
\def\g{\mathfrak{g}}

\def\nh{{\mbox{\tiny{nh}}}}

\def\subW{{\mbox{\tiny{$\W$}}}}
\def\subC{{\mbox{\tiny{$\C$}}}}

\def\subM{{\mbox{\tiny{$\M$}}}}
\def\subSS{{\mbox{\tiny{$S$}}}}
\def\subWW{{\mbox{\tiny{$W$}}}}
\def\subQQ{{\mbox{\tiny{$Q$}}}}

\def\vecOm{\boldsymbol{\Omega}}

\newcommand{\SO}{{\rm SO}}
\def\so{\mathfrak{so}}

\def\vecL{\boldsymbol{\lambda}}

\def\vecgamma{\boldsymbol{\gamma}}
\def\vecalpha{\boldsymbol{\alpha}}
\def\vecbeta{\boldsymbol{\beta}}

\begin{document}

\allowdisplaybreaks

\renewcommand{\thefootnote}{$\star$}

\newcommand{\arXivNumber}{1510.08314}

\renewcommand{\PaperNumber}{018}

\FirstPageHeading

\ShortArticleName{A Geometric Characterization of Certain First Integrals for Nonholonomic Systems}

\ArticleName{A Geometric Characterization of Certain First\\ Integrals for Nonholonomic Systems with Symmetries\footnote{This paper is a~contribution to the Special Issue
on Analytical Mechanics and Dif\/ferential Geometry in honour of Sergio Benenti.
The full collection is available at \href{http://www.emis.de/journals/SIGMA/Benenti.html}{http://www.emis.de/journals/SIGMA/Benenti.html}}}

\Author{Paula BALSEIRO~$^\dag$ and Nicola SANSONETTO~$^\ddag$}

\AuthorNameForHeading{P.~Balseiro and N.~Sansonetto}

\Address{$^\dag$~Universidade Federal Fluminense, Instituto de Matem\'atica,\\
\hphantom{$^\dag$}~Rua Mario Santos Braga S/N, 24020-140, Niteroi, Rio de Janeiro, Brazil}
\EmailD{\href{mailto:pbalseiro@vm.uff.br}{pbalseiro@vm.uff.br}}

\Address{$^\ddag$~Universit\`a degli Studi di Padova, Dipartimento di Matematica,\\
\hphantom{$^\ddag$}~via Trieste 64, 35121 Padova, Italy}
\EmailD{\href{mailto:nicola.sansonetto@gmail.com}{nicola.sansonetto@gmail.com}}

\ArticleDates{Received October 29, 2015, in f\/inal form February 12, 2016; Published online February 21, 2016}

\Abstract{We study the existence of f\/irst integrals in nonholonomic systems with sym\-met\-ry.
  First we def\/ine the concept of  $\mathcal{M}$-cotangent lift of a vector f\/ield on a manifold~$Q$
  in order to unify the works [Balseiro P.,  \textit{Arch. Ration. Mech. Anal.}
  \textbf{214} (2014), 453--501, arXiv:1301.1091],
[Fass{\`o} F., Ramos A., Sansonetto N., \textit{Regul.
  Chaotic Dyn.} \textbf{12} (2007), 579--588], and
 [Fass{\`o} F., Giacobbe A., Sansonetto N., \textit{Rep. Math. Phys.}
  \textbf{62} (2008), 345--367].
  Second, we study gauge symmetries and gauge momenta, in the cases in which
  there are the symmetries that satisfy the so-called {\it vertical symmetry condition}.
  Under such condition we can predict the number of linearly independent
  f\/irst integrals (that are gauge momenta). We illustrate the theory with two examples.}


\Classification{70F25; 70H33; 53D20}

\begin{flushright}
  {\it Dedicated to Sergio Benenti on the occasion of his 70th birthday}
\end{flushright}

\renewcommand{\thefootnote}{\arabic{footnote}}
\setcounter{footnote}{0}

\section{Introduction} \label{S:Intro}
It is well known that for  Hamiltonian systems with a group of symmetries the
components of the momentum map are f\/irst integrals of the system.
For nonholonomic systems the situation is dif\/ferent:
symmetries do not necessarily give rise to f\/irst integrals. Nevertheless the search for f\/irst
integrals is a central topic in the study of nonholonomic systems, and dif\/ferent ideas
to link them to the presence of symmetries have been proposed
and investigated
\cite{BGMConservation, BKMM,BoMa2001,BoMa2008, spagnoli,Crampin-Mestdag,FSG2008, FSG2009,FSG2012,FRS2007, FS2015,Giachetta,Jotz-Ratiu,Marle95, sniatycki, Zenkov}.
In this note we give a geometric interpretation of the notions
of gauge symmetries\footnote{It is worth noting that in this paper~-- as well as in~\cite{FSG2008,FSG2009,FSG2012}~--  the term {\it gauge symmetry} is related to the
search of f\/irst integrals induced by the action of a Lie group and this concept has no
connection with the term {\it gauge transformation} used in \cite{Paula} to ``deform''
brackets.} and gauge momenta introduced in \cite{BGMConservation} and further
developed in \cite{BoMaKi2002, Cortes, FSG2008, FSG2009, FSG2012}, inspired by the tools used in
\cite{Paula} to study the
Hamiltonization problem of nonholonomic systems.

We start by considering a $G$-invariant nonholonomic system on a manifold $Q$
given by a~Lag\-rangian $L\colon TQ\to \R$ and a constraint distribution~$D$; we let
$\M:= \L(D)$ be the constraint submanifold of~$T^*Q$, where $\L$ is the Legendre
transformation (see, e.g., \cite{Benenti, BlochBook, CushmanBook, Marle95, NF, Pars}).
For each section~$\xi$ of the Lie algebra bundle $\g\times Q \to Q$ we consider the
function $\J_\xi\colon \M \to \R$ given by $\J_\xi := {\bf i}_{\xi_\subM} \Theta_\subM$,
where $\Theta_\subM$ is the canonical 1-form on $T^*Q$ restricted to the
submani\-fold~$\M$ and $\xi_\subM$ the inf\/initesimal generator of $\xi$ on $\M$.
We call $\xi\in \g\times Q$ a {\it gauge symmetry} and~$\J_\xi$ a {\it gauge
momentum} if the cotangent lift $\xi_\subQQ^{T^*Q}$ of the inf\/initesimal generator
$\xi_\subQQ$ of~$\xi$ leaves the Hamiltonian $H$ invariant on~$\M$. If $\xi_\subQQ$ is
also a section of the constraint distribution then~$\J_\xi$ is a f\/irst integral of
the nonholonomic dynamics, called {\it horizontal gauge momentum} and~$\xi$ is the {\it horizontal gauge symmetry} that generates it.

In this paper, we introduce the notion of {\it $\M$-cotangent lift} of a vector f\/ield
on $Q$ and show that searching for a horizontal gauge symmetry amounts to
looking for a~section~$\xi$ of $\g\times Q \to Q$  such that the $\M$-cotangent
lift of its inf\/initesimal generator $\xi_\subQQ$ leaves the hamiltonian function invariant.
Following~\cite{Paula}, we observe  that, upon suitable assumptions,  the choice
of a vertical complement~$W$ of the constraint distribution~$D$ induces
a splitting of~$\g\times Q \to Q$ in such a way that horizontal gauge symmetries~$\xi$ are geometrically characterized as projections of constant sections.
This geometric description of  horizontal gauge symmetries is alternative
to the dif\/ferential equations considered in~\cite{BGMConservation, FSG2008}.
Moreover, if~$k$ is the rank of the distribution given by the intersection of~$D$ with the vertical space (the tangent space to the $G$-orbits), then
we show that there are~$k$ linearly independent horizontal gauge momenta.

The paper is organized as follows. In Section~\ref{S:Preliminaries} we recall
some facts about nonholonomic systems with symmetry and gauge momenta.
In Section~\ref{S:GaugeSym}, we characterize horizontal gauge symmetries
with respect to the $\M$-cotangent lift and then, using a vertical complement
of the constraints, we describe the horizontal gauge momenta in terms of Lie
algebra elements. Section~\ref{S:Examples} illustrates the theory with two
examples. We f\/inish with a section of conclusions and perspectives.

Throughout the paper all manifolds, distributions and maps are smooth; the distributions
are assumed to be regular and the group actions are free and proper.

\section[Preliminaries: conserved quantities in nonholonomic mechanics]{Preliminaries: conserved quantities\\ in nonholonomic mechanics} \label{S:Preliminaries}

\subsection{Nonholonomic systems with symmetry}\label{Ss:NH}

\looseness=-1
{\bf Nonholonomic systems.} Let us denote by $(Q, L, D)$
a nonholonomic system on a manifold $Q$ def\/ined by a Lagrangian
$L\colon TQ\to \R$ of mechanical type and a (non-integrable) distribution~$D$ on~$Q$ describing the permitted velocities.
The {\it constraint submanifold} $\M\subset T^*Q$ is def\/ined by
$\M:=\kappa^\flat(D)$, where $\kappa$ is the kinetic energy metric
that induces the isomorphism $\kappa^\flat\colon TQ\to T^*Q$
def\/ined by $\kappa^\flat(X)(Y)= \kappa(X,Y)$ for $X,Y \in TQ$.
Note that since $\kappa$ is linear on the f\/ibers, then $\M$ is a vector subbundle of
$T^*Q$; we denote by $\tau_\subM\colon \M \to Q$ the canonical projection.
If $\Ham\colon T^*Q\to \R$ is the Hamiltonian function associated to~$L$,  then we denote
by $\Ham_\subM \colon \M \to R$ the pull-back of
the Hamiltonian $\Ham$ to $\M$, i.e., $\Ham_\subM = \iota^*\Ham$,
where $\iota \colon \M \to T^*Q$ is the natural inclusion.
Moreover, we def\/ine $\Omega_\subM$ the 2-form on $\M$ given by
$\Omega_\subM := \iota^*\Omega_\subQQ$, where $\Omega_\subQQ$ is the
canonical 2-form on $T^*Q$.
Let $\C$ be the non-integrable distribution on~$\M$ given, at each $m\in \M$, by
\begin{gather*}%\label{Def:C}
\C_m :=\{ v_m \in T_m\M \colon T\tau_\subM(v_m) \in D_{\tau_\M(m)} \}.
\end{gather*}

The nonholonomic dynamics is described by the integral curves of the vector
f\/ield $X_\nh$ on $\M$ def\/ined by
\begin{gather}\label{Eq:Dyn}
{\bf i}_{X_\nh} \Omega_\subM |_\C = d\Ham_\subM |_\C  ,
\end{gather}
where $\Omega_\subM|_\C$ denotes the restriction of $\Omega_\subM$ to $\C$.
Since the vector f\/ield $X_\nh$ takes
values on $\C$, we say that it is  a section of the bundle $\C\to\M$, i.e.,
$X_\nh\in \Gamma(\C)$.  It is important to note that the solution
$X_\nh$ satisfying~\eqref{Eq:Dyn} is unique since the 2-section $\Omega_\subM |_\C$
is nondegenerate~\cite{BS93}. The vector f\/ield $X_\nh$ is also called the {\it nonholonomic dynamics}
or the {\it nonholonomic vector field}.

In this setting, a function $f$ on $\M$ is a {\it first integral} of $X_\nh$ (or $f$ is {\it conserved by}
$X_\nh$) if and only if $X_\nh(f) = 0$.  From \eqref{Eq:Dyn} $f \in C^\infty(\M)$ is a f\/irst integral
if and only if $X_f(\Ham_\subM)=0$, where $X_f \in \mathfrak{X}(\M)$ is the
{\it nonholonomic Hamiltonian vector field}\footnote{$X_f$
is also called distributional Hamiltonian vector f\/ield, see~\cite{CushmanBook}.} on $\M$ def\/ined by
\begin{gather}\label{Eq:HamVectorField}
{\bf i}_{X_f} \Omega_\subM |_\C = df |_\C  .
\end{gather}
Observe that the nonholonomic Hamiltonian vector f\/ield $X_f$ given in \eqref{Eq:HamVectorField} is a section of $\C\to \M$, i.e., $X_f \in \Gamma(\C)$.
Moreover, since $\Omega_\subM|_\C$ is nondegenerate, equation \eqref{Eq:HamVectorField}  induces an almost Poisson bracket $\{\cdot, \cdot \}_\nh$ on $\M$ called the {\it nonholonomic bracket} \cite{IbLeMaMa1999,Marle1998, SchaftMaschke1994}.  In fact,  for every $f,g \in C^\infty(\M)$ the nonholonomic bracket $\{\cdot, \cdot \}_\nh$ is def\/ined by  $\{f,g\}_\nh = -X_f(g)$, where $X_f$ is the unique vector f\/ield on $\M$ (with values in $\C$) that satisf\/ies \eqref{Eq:HamVectorField}. Hence, $f\in C^\infty(\M)$ is a~f\/irst integral of the nonholonomic dynamics if and only if $\{f,\Ham_\subM\}_\nh = 0$, in which case we say that the functions $f$
and $\Ham_\subM$ are in {\it involution}.

As it is well known there is a correspondence between
almost Poisson brackets $\{\cdot,\cdot\}$ on a~mani\-fold $\M$ and bivector f\/ields
$\pi\in\Gamma\left( \Lambda^2 (T\M\right))$, def\/ined by $\pi(df,dg) = \{f,g\}$,
for $f,g\in C^{\infty}(\M)$  (see, e.g.,~\cite{MarsdenRatiuBook}). Throughout the paper we will work indistinguishably with almost
Poisson brackets and bivector f\/ields. We denote by $\pi^\sharp\colon T^*\M\longrightarrow T\M$ the map
such that $\beta(\pi^\sharp(\alpha)) = \pi(\alpha,\beta)$.
In our context, the nonholonomic bracket $\{\cdot,\cdot\}_\nh$ induces the
{\it nonholonomic bivector field} $\pi_\nh$ and the nonholonomic dynamics is
given by the vector f\/ield $X_\nh = - \pi^\sharp_\nh(d\Ham_\subM)$.

{\bf Symmetries.}
Let $G$ be a Lie group acting on $Q$ freely and properly.  We say that the $G$-action is a symmetry
of the nonholonomic system $(Q,L,D)$ if the tangent lift of the action to $TQ$ leaves
$L$ and $D$ invariant.  In this case, the cotangent lift of the $G$-action to $T^*Q$
leaves the constraint submanifold $\M$ invariant  and thus it induces, by restriction, a~$G$-action on~$\M$:
\begin{gather*}
\phi\colon \ G \times \M \to \M.
\end{gather*}
It is straightforward to see that the Hamiltonian $\Ham_\subM$ and the nonholonomic
bracket $\{\cdot, \cdot \}_\nh$ are $G$-invariant~\cite{IbLeMaMa1999,Marle1998}.

Let us denote by $V$ the distribution on $Q$ whose f\/ibers $V_q$, for $q\in Q$, are the tangent
spaces to the orbits of $G$ in $Q$, that is $V_q = T(\operatorname{Orb}_q(q))$. Equivalently,
$\V$ is the distribution on $\M$ such that for each $m\in \M$ the f\/iber $\V_m$
is the tangent space to the orbit of $G$ in $\M$.
Since the action on $\M$ is lifted from an action on $Q$ then, for each $m\in \M$,
the f\/ibers $\V_m$ and $V_{\tau_\M(m)}$ are dif\/feomorphic. We call $V$ and $\V$
the {\it vertical spaces} associated to the $G$-action on~$Q$ and on~$\M$, respectively.
For each $\eta\in \mathfrak{g}$, where $\mathfrak{g}$ denotes the Lie algebra of~$G$,
$\eta_{\mbox{\tiny{$Q$}}}$ and $\eta_{\mbox{\tiny{$T^*Q$}}}$ denote the inf\/initesimal
generator associated to the $G$-action on $Q$ and on $T^*Q$, respectively.
Since $\M$ is an invariant submanifold of $T^*Q$, for each $\eta\in\g$ the
inf\/initesimal generator~$\eta_\subM$ of~$\eta$ on $\M$ is well def\/ined;  $\eta_{\mbox{\tiny{$T^*Q$}}} (m) = \eta_\subM (m) \in T_m\M$, for each $m\in\M$.

Consider the distributions $S$ and $\S$ on $Q$ and $\M$, respectively,
given by
\begin{gather*}
  S_q = D_q \cap V_q \qquad \mbox{and} \qquad \S_m = \C_m\cap\V_m,
\end{gather*}
for $q\in Q$ and $m\in \M$.
The distribution $S$
induces a vector subbundle  $\mathfrak{g}_\subSS \to Q$
of the trivial vector bundle $\mathfrak{g}\times Q \to Q$, def\/ined by
\begin{gather} \label{Def:S}
 \mathfrak{g}_\subSS = \{ \xi _q \in \mathfrak{g}\times Q \colon \xi_{\mbox{\tiny{$Q$}}}(q) \in S_q\},
 \end{gather}
where $\xi_\subQQ(q) := (\xi_q)_\subQQ (q)$.  Observe that, for
$\xi\in \Gamma(\mathfrak{g} \times Q)$, the corresponding inf\/initesimal generator
$\xi_\subM$ of the lifted action on $\M$ is def\/ined by
$\xi_\subM (m) := (\xi_q)_\subM(m)\in \Gamma(\V)$ having in mind that, for a f\/ixed $q\in Q$, $\xi_q\in\g$.

The {\it nonholonomic momentum map} \cite{BKMM} ${\mathcal J}^{\nh} \colon \M \to \mathfrak{g}_\subSS^*$
is given\footnote{$\mathfrak{g}^*_\subSS$ denotes the dual vector bundle of $\mathfrak{g}_\subSS$.},
at each $m \in \M$ and $\xi \in \Gamma(\mathfrak{g}_\subSS)$, by
\begin{gather}\label{Eq:NH-momentum}
\langle {\mathcal J}^{\nh}(m) , \xi (m) \rangle := {\bf i}_{\xi_\M} \Theta_\subM (m)  ,
\end{gather}
where $\Theta_\subM$ is the pull-back to $\M$ of the canonical 1-form
$\Theta_{\mbox{\tiny{$Q$}}}$ on $T^*Q$ via the natural inclu\-sion~$\iota$, i.e.,
$\Theta_\subM := \iota^*\Theta_{\mbox{\tiny{$Q$}}}$, see \cite{Paula, BKMM}.

\begin{Remark}
 The nonholonomic momentum map can be evaluated at non-constant sections
 of the bundle $\mathfrak{g}\times \M \to \M$.  If there is a constant section
 $\eta \in \mathfrak{g}$ such that $\eta \in \Gamma(\g_\subSS)$, then $\langle \J^{\nh} , \eta \rangle = \langle {\mathcal J} ,
 \eta \rangle $, where ${\mathcal J}\colon \M \to \mathfrak{g}^*$ is the restriction to $\M$
 of the canonical momentum map def\/ined on~$T^*Q$.
\end{Remark}

As it is already well known \cite{BGMConservation,Cortes,FSG2008}, the nonholonomic momentum
map gives a candidate for a f\/irst integral of a nonholonomic system with symmetry.
In fact, by \eqref{Eq:NH-momentum}, for $\xi\in\Gamma(\g_\subSS)$,
\begin{gather}\label{Eq:ConservedNHmom}
X_\nh(\langle {\mathcal J}^{\nh} , \xi  \rangle) = \pounds_{\xi_\M}  \Theta_\subM (X_\nh) + \Omega_\subM(\xi_\subM, X_\nh).
\end{gather}
Since $\xi_\subM\in \Gamma(\S)$ then $\Omega_\subM(\xi_\subM, X_\nh)\! =\!
- d\Ham_\subM (\xi_\subM)\! =\! 0$, where in the last equality we use the
$G$-in\-va\-riance of the Hamiltonian $\Ham_\subM$. We conclude that for
a section $\xi\in\Gamma(\g_\subSS)$, the function $\langle {\mathcal J}^{\nh} , \xi \rangle$
is constant along the f\/low of the nonholonomic dynamics if and only if
\mbox{$(\pounds_{\xi_\M}\Theta_\subM) (X_\nh)  = 0$}.

\subsection{Conserved quantities and cotangent lifts}\label{Ss:Cotangent}

Given a vector f\/ield $X$ on $Q$, we denote by $X^{T^*Q}$ the cotangent lift of $X$ to $T^*Q$.
Recall that the cotangent lift of $X$ is the inf\/initesimal generator of $T\Phi_t^X\colon T^*Q \to T^*Q$,
where $T\Phi^X_t$ is the lift to $T^*Q$ of the f\/low $\Phi_t^X$ of $X$ with respect
to the canonical projection $\tau_\subQQ\colon T^*Q\longrightarrow Q$.
In other words, the cotangent lift of $X$ can be expressed
as a Hamiltonian vector f\/ield with respect to the canonical symplectic structure on $T^*Q$.
That is, given $X\in \mathfrak{X}(Q)$ we denote by $\tilde{X}$ the linear function on $T^*Q$ given by
$\tilde{X} (\alpha) = \alpha(X)$ for $\alpha\in T^*Q$, then
$X^{T^*Q} = -\pi_{\mbox{\tiny{$Q$}}}^\sharp(d\tilde{X})$, where~$\pi_{\mbox{\tiny{$Q$}}}$
is the bivector f\/ield associated to the canonical symplectic form $\Omega_{\mbox{\tiny{$Q$}}}$.
In local coordinates, if~$\{X_i\}$ is a~local basis of $TQ$ and $(q,p_i)$
denote the local coordinates in $T^*Q$ associated to the dual basis $\{X^i\}$, then
$X_ i^{T^*Q} = - \pi_{\mbox{\tiny{$Q$}}}^\sharp(dp_i)= X_i + p_kC_{ij}^k\partial_{p_j},$
where $C_{ij}^k\in C^\infty(Q)$ are the structure functions $[X_i,X_j] = C_{ij}^k X_k$. If $X=f_i X_i$ for $f_i\in C^\infty(Q)$, then
\begin{gather*}
X^{T^*Q} = f_i^v X_i^{T^*Q} - p_i\pi_\subQQ^\sharp(df_i) = f^v_i X_i^{T^*Q} - p_iX_j(f_i)\partial_{p_j},
\end{gather*}
where $f^v_i = \tau_{\mbox{\tiny{$Q$}}}^*f$ for $\tau_{\mbox{\tiny{$Q$}}} \colon T^*Q\to Q$, see~\cite{MarsdenRatiuBook}.

In this paper, we will focus on the cotangent lift $\xi_{\mbox{\tiny{$Q$}}}^{T^*Q}$
of inf\/initesimal generators of sections $\xi$ of $\Gamma(\g \times Q)$.
It is important to note that if  $\eta\in \g$ (i.e., $\eta$ is a constant section of
$\g\times Q\to Q$) then $\eta_{\mbox{\tiny{$Q$}}}^{T^*Q} = \eta_{\mbox{\tiny{$T^*Q$}}}$.
However this is not necessarily true if we consider non-constant sections
$\xi\in \Gamma(\g\times Q)$, indeed the lift of a section in $\Gamma(V)$
might not be tangent to the orbits of the lifted action to $T^*Q$.
On the other hand, for a section $\xi\in \Gamma(\g\times Q)$ we have that the corresponding linear function of $\xi_\subQQ$ is given by
$\tilde{\xi_\subQQ} = {\bf i}_{\xi_{T^*Q}} \Theta_\subQQ = \langle J, \xi\rangle$
and thus
\begin{gather} \label{Eq:lift}
 \xi_\subQQ^{T^*Q} = - \pi_\subQQ^\sharp (d \langle J, \xi \rangle).
\end{gather}

Following \cite{FSG2008,FSG2009} we say that a {\it gauge symmetry}
of a nonholonomic system $(Q,L,D)$ with symmetry $G$, is a
section $\xi$ of $\g\times Q\to Q$ such that the cotangent lift
$\xi_{\mbox{\tiny{$Q$}}}^{T^*Q}$ of the inf\/initesimal generator
$\xi_{\mbox{\tiny{$Q$}}}$ leaves the Hamiltonian $\Ham$ invariant on $\M$, i.e.,
$\xi_{\mbox{\tiny{$Q$}}}^{T^*Q}(\Ham)(m) = 0$, $\forall\, m\in \M$.
The {\it gauge momentum}
${\mathcal J}_\xi\colon \M\to \R$ of a gauge symmetry $\xi$ is the function
 \begin{gather}\label{Def:GaugeMomentum}
    {\mathcal J}_\xi := \iota^*(i_{\xi^{T^*Q}_{\mbox{\tiny{$Q$}}}}  \Theta_{\mbox{\tiny{$Q$}}})  .
 \end{gather}

Next proposition clarif\/ies the link between gauge momenta and f\/irst integrals

\begin{Proposition}[\cite{FRS2007}]\label{Prop:FRS}
 Consider a nonholonomic system $(Q,L,D)$ with symmetry $G$ and a~section
 $\xi\in \Gamma(\g_\subSS)$. Then $\xi_{\mbox{\tiny{$Q$}}}^{T^*Q} (\Ham) (m) = 0$, $\forall\, m\in\M$,  if and only if
 ${\mathcal J}_\xi$ is a~first integral of the nonholonomic dynamics $X_{\emph\nh}$.
\end{Proposition}

When a gauge symmetry $\xi$ is also a section of $\g_\subSS$
we say that $\xi$ and, with a slight abuse, $\xi_{\mbox{\tiny{$Q$}}}$
are {\it $D$-gauge symmetries} or {\it horizontal gauge symmetries}.
In this case, the f\/irst integral $\J_\xi$ is a~{\it $D$-gauge momentum} or an
{\it horizontal gauge momentum}. If $\xi \in \Gamma( \mathfrak{g}_\subSS)$
is a constant section (i.e., $\xi \in \mathfrak{g}$ such that
$\xi_{\subQQ} \in \Gamma(D)$), then $\xi$ is called a~{\it $D$-symmetry} or an
{\it horizontal symmetry}~\cite{BKMM,FSG2008}.

\begin{Remarks}\quad
  \begin{enumerate}\itemsep=0pt
    \item[$(i)$] Observe that the nonholonomic momentum \eqref{Eq:NH-momentum}
    and the gauge momentum \eqref{Def:GaugeMomentum} coincide. In fact,
    for $\xi\in \Gamma(\g_\subSS)$ and for all $m\in \M$,
    \begin{gather*} \langle \J^\nh, \xi\rangle (m)  =
      {\bf i}_{\xi_\subM}
      \Theta_\subM (m) = \langle \iota^*m,  \xi_\subQQ \rangle =  \big\langle \iota^* m,
      T\tau_\subQQ \big( \xi_\subQQ^{T^*Q} \big) \big\rangle \\
      \hphantom{\langle \J^\nh, \xi\rangle (m) }{}
      =  \iota^*({\bf i}_{\xi^{T^*Q}_{\mbox{\tiny{$Q$}}}}
      \Theta_{\mbox{\tiny{$Q$}}}) (m)= \J_\xi(m).
    \end{gather*}
    \item[$(ii)$] In \cite{FSG2012} the def\/inition of gauge momenta is given without the
    hypothesis of invariance of the Hamiltonian. However, in this paper the cases
    of interest are the ones in which the Hamiltonian is invariant.
    For a discussion of this see~\cite{FSG2012}.
  \end{enumerate}
\end{Remarks}

\section[Gauge-symmetries and the relation with a vertical complement of the constraints]{Gauge-symmetries and the relation\\ with a vertical complement of the constraints}\label{S:GaugeSym}

\subsection[The $\M$-cotangent lift]{The $\boldsymbol{\M}$-cotangent lift}

In this section, we study horizontal gauge symmetries in terms of vector f\/ields
and functions on~$\M$, in order to unify the viewpoints of
Sections~\ref{Ss:NH} and~\ref{Ss:Cotangent}.

Following \cite{BS93,CushmanBook} we recall that, for all $m\in\M$
\begin{gather} \label{Eq:C+Comega}
T_m(T^*Q) = \C_m  \oplus {\C}_m^\Omega,
\end{gather}
where ${\C}^\Omega_m := \{ v_m \in T_m(T^*Q) \colon \Omega_{\mbox{\tiny{$Q$}}}
(v_m , w_m) = 0 \ \forall\, w_m \in {\C}_m\}$ is the symplectic orthogonal to $\C_m$.
If we denote by ${\mathcal T}_\subC(m) \colon T_m(T^*Q) \to \C_m$ the projection associated
to decomposition~\eqref{Eq:C+Comega}, then to every $X\in \mathfrak{X}(T^*Q)$
we can associate a vector f\/ield on $\M$ with value on $\C$ by
$\mathcal{T}_\C(X)\in \Gamma(\C)$. Moreover, for each $m\in \M$,
\begin{gather}\label{eq:field}
\pi_\nh^\sharp(df)(m) =  {\mathcal T}_\subC \big(\pi_\subQQ^\sharp\big(d \tilde f\big)\big)(m),
\end{gather}
where $f\in C^\infty(\M)$ and $\tilde f\in C^\infty(T^*Q)$ is an extension of~$f$,
i.e., $\tilde f|_\subM = f$.
Observe that~(\ref{eq:field}) is independent from the choice of the extension,
see \cite{IbLeMaMa1999, sniatycki}.

\begin{Definition}\label{Def:M-lift}
 For a section $\xi\in \Gamma(\g\times Q)$, the {\it $\M$-cotangent lift} of
 $\xi_\subQQ\in \mathfrak{X}(Q)$ is the vector f\/ield $\xi_\subQQ^\subM$ on $\M$
 with values on $\C$ def\/ined, at each $m\in\M$, by
 \begin{gather*}
 \xi_\subQQ^\subM (m) := {\mathcal T}_\subC \big(\xi_\subQQ^{T^*Q} \big)(m).
 \end{gather*}
\end{Definition}

\begin{Lemma}\label{L:C-lift}
If $\xi\in\Gamma(\g_\subSS)$ is a section given by $\xi= f_i   \eta_i$,\footnote{Observe that $V$ is a regular distribution since the $G$-action is locally free \cite{Ortega-Ratiu}, then each section $\xi\in \Gamma(\g\times Q)$ can be written as linear combinations of elements in $\g$.} where $f_i \in C^\infty(Q)$ and $\eta_i \in \g$ then,
\begin{enumerate}\itemsep=0pt
 \item[$(i)$] $\xi_\subQQ^\subM = \xi_{\subM}  + \langle \J,\eta_i\rangle   \pi_{\emph\nh}^\sharp(df_i) $ where, in this case, $f_i$ represents the pull back of the func\-tions~$f_i$ to~$\M$;
 \item[$(ii)$] $\xi_\subQQ^\subM  = - \pi_{\emph\nh}^\sharp(d\langle \J^{\emph\nh}, \xi\rangle)$.
\end{enumerate}
\end{Lemma}

\begin{proof} To see $(i)$ observe that $\xi_\subQQ^\subM (m) = {\mathcal T}_\subC(\xi_{\mbox{\tiny{$T^*Q$}}} )(m)   + \langle J,\eta_i\rangle {\mathcal T}_\subC (\pi^\sharp_{\mbox{\tiny{$Q$}}}(df_i) )(m) $ where we keep the notation $f_i$ to denote the pull back of the functions $f_i$ to $T^*Q$.  Item~$(ii)$ is a consequence of~\eqref{Eq:lift} and Def\/inition~\ref{Def:M-lift}.
\end{proof}

Next, we give another insight of Proposition~\ref{Prop:FRS} but in terms of
$\M$-cotangent lifts and the nonholonomic momentum map and we prove
it using the geometrical approach of Section~\ref{Ss:NH}.

\begin{Proposition}\label{Prop:3Conditions}
 Given a nonholonomic system $(Q,L,D)$ and a section $\xi\in \Gamma(\g_\subSS)$, then
 \begin{enumerate}\itemsep=0pt
  \item[$(i)$]  $\pounds_{\xi_\M} \Theta_\subM (X_{\emph\nh})  = \xi_{\mbox{\tiny{$Q$}}}^\subM (\Ham_\subM)$;
 \item[$(ii)$] The function $\langle{\mathcal J}^{\emph\nh}, \xi\rangle$
 is a first integral of the nonholonomic dynamics $X_{\emph\nh}$
 if and only if $\xi_{\mbox{\tiny{$Q$}}}^\subM (\Ham_\subM) = 0$.
 \end{enumerate}
\end{Proposition}

\begin{proof}
To prove $(i)$,  f\/irst observe that $\pounds_{\xi_\M}\Theta_\subM  = f_i   \pounds_{{(\eta_i)}_\M} \Theta_\subM + \langle{\mathcal J}, \eta_i \rangle df_i= \langle{\mathcal J}, \eta_i \rangle df_i$ by the $G$-invariance of $\Theta_\subM$. Then we obtain that
\begin{gather*}
\pounds_{\xi_\M}\Theta_\subM (X_\nh) = \langle{\mathcal J}, \eta_i \rangle X_\nh(f_i) = \langle{\mathcal J}, \eta_i \rangle \pi^\sharp_\nh(df_i) (\Ham_\subM).
\end{gather*}
On the other hand, using Lemma \ref{L:C-lift}$(i)$ and since the hamiltonian $\Ham_\subM$ is $G$-invariant, we have that $\xi_{\mbox{\tiny{$Q$}}}^{\subM} (\Ham_\subM) =   \langle \J,\eta_i\rangle    \pi_\nh^\sharp(df_i)  (\Ham_\subM)$ and hence $\pounds_{\xi_\M} \Theta_\subM (X_{\nh})  = \xi_{\mbox{\tiny{$Q$}}}^\subM (\Ham_\subM)$.

Item $(ii)$ is a trivial consequence of $(i)$ since, as already seen in \eqref{Eq:ConservedNHmom}, for $\xi\in\Gamma(\g_\subSS)$ the function ${\mathcal J}_\xi = \langle{\mathcal J}^{\nh}, \xi\rangle$ is conserved if and only if $\pounds_{\xi_\M} \Theta_\subM (X_\nh) = 0$, which is equivalent to ask for $\xi_{\mbox{\tiny{$Q$}}}^\subM (\Ham_\subM)=0$.
\end{proof}

\begin{Remark}\label{R:Cartan}
Given a symplectic manifold $(P, \Omega)$, a Cartan symmetry $Z\in \mathfrak{X}(P)$ of a Hamiltonian vector f\/ield~$X_H$ is a vector f\/ield satisfying $Z(H)=0$ and $\pounds_{Z}\Omega=0$~\cite{Book:Crampin}.  In our setting, if $\xi$ is an horizontal gauge symmetry then $\xi_\subM(\Ham_\subM) = 0$ but $\pounds_{\xi_\subM} \Omega_\subM = d\alpha$, where $\alpha$ is a 1-form on $\M$ such that $\alpha(X_\nh) = 0$. Thus $\xi_\subM\in\mathfrak{X}(\M)$ could be interpreted as a {\it nonholonomic Cartan symmetry} for the nonholonomic system $(\M, \pi_\nh, \Ham_\subM)$.
On the other hand in \cite[Theorem~1]{Crampin-Mestdag} a~Cartan-type symmetry of a nonholonomic
systems is a vector f\/ield that
leaves the Hamiltonian invariant and such that $\alpha(X)=0$ for all $X\in \Gamma(\C)$, while we
ask only for~$\alpha$ to annihilate only~$X_\nh$. In fact, we will see in Example~\ref{Ex:Ball} that the Chaplygin ball satisf\/ies that $\alpha |_\C \neq 0$.
\end{Remark}

Proposition~\ref{Prop:3Conditions} gives a characterization of horizontal gauge symmetries and horizontal gauge momenta in terms of the $\M$-cotangent lifts, so we conclude

\begin{Corollary} \label{C:MLift}
  Consider a nonholonomic system $(Q,L,D)$ with symmetry $G$.
  A section $\xi\in\Gamma(\g_\subSS)$ is a {\it horizontal gauge symmetry} if and only if
  the $\M$-cotangent lift $\xi_{\mbox{\tiny{$Q$}}}^{\subM}$ of the infinitesimal generator
  $\xi_{\mbox{\tiny{$Q$}}}$ leaves the Hamiltonian $\Ham_\subM$ invariant. In this case, the associated horizontal gauge momentum is the first integral ${\mathcal J}_\xi\colon \M\to \R$ given by
  \begin{gather*}%\label{Eq:M-gaugeMomentum}
  {\mathcal J}_\xi:=
  {\bf i}_{\xi^\subM_{\mbox{\tiny{$Q$}}}}  \Theta_\subM = \langle \J^{\emph\nh}, \xi\rangle .
  \end{gather*}
\end{Corollary}

\begin{proof}
From Propositions~\ref{Prop:FRS} and~\ref{Prop:3Conditions} we see that a section $\xi\in\Gamma(\g_\subSS)$ satisf\/ies that  $\xi_\subQQ^{T^*Q} (\Ham) (m)$ $ = 0$ for all $m\in\M$, if and only if $\xi_\subQQ^\subM (\Ham_\subM) = 0$. Lemma~\ref{L:C-lift} ensures that the function $\mathcal{J}_\xi = {\bf i}_{\xi^\subM_{\mbox{\tiny{$Q$}}}}  \Theta_\subM$ coincides with the def\/inition given
 in~\eqref{Def:GaugeMomentum} and of course with the nonholonomic momentum map~\eqref{Eq:NH-momentum}. In fact, since $T\tau_\subM ( \xi_\subQQ^\subM ) = \xi_\subQQ$ then,
 \begin{gather*}
 \J_\xi(m) = {\bf i}_{\xi^\subM_{\mbox{\tiny{$Q$}}}}
 \Theta_\subM (m) = \langle \iota^*m, T\tau_\subM ( \xi_\subQQ^\subM ) \rangle =  \langle \iota^*m,
  \xi_\subQQ \rangle =  \langle {\mathcal J}^\nh , \xi \rangle (m).\tag*{\qed}
 \end{gather*}
 \renewcommand{\qed}{}
\end{proof}

From Proposition~\ref{Prop:3Conditions} we see that the problem of f\/inding a f\/irst integral of the nonholonomic dynamics~-- that is a horizontal gauge momentum~-- is reduced to  f\/inding a section $\xi \in \Gamma(\g_\subSS)$ such that $\xi^\subM_\subQQ(\Ham_\subM) = 0$. Still, the choice of such a section $\xi \in \Gamma(\g_\subSS)$ is not canonical.
However, it was seen in \cite{Paula} that, under certain circumstances, there is a ``natural'' choice of a section
$\xi \in \Gamma(\mathfrak{g}_{\subSS})$ so that $\J_\xi$ is a f\/irst integral.
This section $\xi$ depends on the choice of a {\it vertical complement of the constraints}.
We will investigate and relate these view--points in the next section.

\subsection{Splitting adapted to the constraints and the vertical symmetry condition}

Consider a nonholonomic system $(Q,L,D)$ with a $G$-symmetry.
We say that the $G$-symmetry verif\/ies the {\it dimension assumption} if
\begin{gather} \label{Eq:DimAssumptionTQ}
 T_qQ = D_q +V_q  , \qquad \mbox{for} \ q \in Q,
\end{gather}
or equivalently if
\begin{gather*}
  T_m\M = \C_m + \V_m, \qquad \textrm{for} \ m\in \M.
\end{gather*}

Let us consider a $G$-invariant {\it vertical complement $W$ of the constraints}
$D$, that is, a $G$-invariant distribution on $Q$ such that for each $q \in Q$,
\begin{gather}\label{Eq:VertComplementTQ}
T_qQ = D_q \oplus W_q, \qquad \mbox{and} \qquad W_q \subset V_q.
\end{gather}
Let now $\W$ be the distribution on $\M$ given, at each $m\in \M$, by
$\W_m := (T\tau_\subM|_\V)^{-1}(W_{\tau(m)})$, where $T\tau_\subM|_\V \colon \V \to V$
is the restriction of $T\tau_\subM$ to $\V\subset T\M$. Then $\W$ is a $G$-invariant
{\it vertical complement} of $\C$ since, for each $m\in\M$,
\begin{gather}\label{Eq:VertComplementTM}
T_m\M = \C_m \oplus \W_m, \qquad \mbox{and} \qquad \W_m \subset \V_m.
\end{gather}
Splitting \eqref{Eq:VertComplementTQ} induces a splitting of~$V$ or equivalently, \eqref{Eq:VertComplementTM} induces a splitting of~$\V$:
\begin{gather}\label{Eq:SplittingS}
V_q = S_q \oplus W_q \qquad \mbox{and} \qquad \V_m = \S_m \oplus \W_m.
\end{gather}
Therefore, the Lie algebra $\g$ admits a splitting induced by \eqref{Eq:SplittingS} given, at each $q\in Q$, by
\begin{gather} \label{Eq:gS+gW}
\mathfrak{g}|_q = \mathfrak{g}_\subSS |_q \oplus \mathfrak{g}_\subWW |_q,
\end{gather}
where  $\mathfrak{g}_\subWW\to Q$ is the vector subbundle of the trivial vector bundle $\mathfrak{g}\times Q \to Q$ def\/ined by
\begin{gather*} %\label{Def:gW}
 \mathfrak{g}_\subWW = \{ \xi _q \in \mathfrak{g}\times Q \colon \xi_\subQQ(q) \in W_q\}.
\end{gather*}
Let us denote by $P_{\mathfrak{g}_S}\colon \mathfrak{g}\to \mathfrak{g}_\subSS$
and $P_{\mathfrak{g}_W}\colon \mathfrak{g}\to \mathfrak{g}_\subWW$  the
projections associated to decomposition~\eqref{Eq:gS+gW}.

\begin{Remark} \label{R:Proj} For every $\eta\in \mathfrak{g}$ we have that $P_{D}(\eta_{\mbox{\tiny{$Q$}}}) = ( P_{\mathfrak{g}_S}(\eta))_{\mbox{\tiny{$Q$}}}$ and  $P_{\subC} (\eta_\subM) = ( P_{\mathfrak{g}_S}(\eta)  )_\subM$,
 where $P_D \colon TQ \to D$ and $P_\C \colon T\M\to \C$ are the projections associated
 to the decompositions~\eqref{Eq:VertComplementTQ} and~\eqref{Eq:VertComplementTM},
 respectively.
\end{Remark}

 \begin{Definition}[\protect{\cite[Section~4.4]{Paula}}] \label{D:VertSym}
We say that $W$ satisf\/ies the
{\it vertical symmetry condition} if $\mathfrak{g}_\subWW$ is a trivial bundle over~$\M$,
that is, if~$\mathfrak{g}_\subWW$ is a~Lie subalgebra of~$\mathfrak{g}$.
\end{Definition}

When the vertical symmetry condition is satisf\/ied there is a natural candidate
to be chosen as horizontal gauge symmetry. In order to properly state next proposition,
we need to introduce a~2-form, denoted by
$\langle \mathcal{J}, \mathcal{K}_\subW \rangle$ (see~\cite{Paula}),
arising from the data that def\/ines a nonholonomic system with a
$G$-symmetry satisfying the dimension assumption \eqref{Eq:DimAssumptionTQ}.
More precisely, a $G$-invariant vertical complement $\W$ induces a $\g$-valued
1-form ${\mathcal A}_\subW\colon  T\M \to \mathfrak{g}$ def\/ined by
${\mathcal A}_\subW (v_m) = \xi$ if and only if $P_\subW (v_m) = \xi_\subM(m)$
for each $m\in\M$, where $P_\subW \colon T\M \to \W$ is the projection associated to decomposition \eqref{Eq:VertComplementTM}. Then, the associated $\W$-{\it curvature} is the $\mathfrak{g}$-valued
2-form on $\M$ given by
\begin{gather*} %\label{Def:KinCurv-2form}
\mathcal{K}_\subW(X,Y):=  d {\mathcal A}_\subW(P_\subC (X),P_\subC (Y)) \qquad \mbox{for} \ X,Y \in T\M.
\end{gather*}
Finally $\langle \mathcal{J}, \mathcal{K}_\subW \rangle$ is the 2-form on $\M$
def\/ined by the natural pairing between $\J \colon \M\to \g^*$~-- the restriction to $\M$ of the canonical momentum map~--
and the $\mathfrak{g}$-valued 2-form $\mathcal{K}_\subW$.

\begin{Proposition}\label{Prop:Conserved}
 Suppose that the $G$-invariant vertical complement $W$ satisfies the
 vertical-symmetry condition. If for $\eta\in \mathfrak{g}$ we have that
 $\langle \mathcal{J},\mathcal{K}_\subW\rangle (X_{\emph\nh},( P_{\mathfrak{g}_S}(\eta))_\subM )=0$
 then
 \begin{enumerate}\itemsep=0pt
  \item[$(i)$] $f_\eta = \langle \J^{\emph\nh}, P_{\mathfrak{g}_S}(\eta) \rangle$ is
  a first integral of the nonholonomic dynamics $X_{\emph\nh}$, which
  is linear in the momenta.
  \item[$(ii)$] The  $\M$-cotangent lift $( P_{\mathfrak{g}_S}(\eta))_{\mbox{\tiny{$Q$}}} ^{\subM} \!\in\! \mathfrak{X}(\M)$    of   $ (P_{\mathfrak{g}_S}(\eta))_{\mbox{\tiny{$Q$}}} \!\in\! \mathfrak{X}(Q)$
  satisfies $(P_{\mathfrak{g}_S}(\eta))_{\mbox{\tiny{$Q$}}} ^{\subM} (\Ham_\subM)\! =\! 0$.
 \end{enumerate}
\end{Proposition}

\begin{proof}
$(i)$ Following \cite[Theorem~6.5(ii)]{Paula} we have that
${\bf i}_{( P_{\mathfrak{g}_S}(\eta))_\subM} \Omega_{\mbox{\tiny{${\mathcal J} {\mathcal K}$}}} |_\C= df_\eta|_\C$  for  $\Omega_{\mbox{\tiny{${\mathcal J} {\mathcal K}$}}} := \Omega_\subM + \langle \mathcal{J},\mathcal{K}_\subW\rangle$.
Then for each $\eta \in \mathfrak{g}$,
\begin{gather*}
X_\nh (f_\eta) = df_\eta (X_\nh) = \Omega_{\mbox{\tiny{${\mathcal J} {\mathcal K}$}}}( (P_{\mathfrak{g}_S}(\eta) )_\subM , X_\nh) \\
\hphantom{X_\nh (f_\eta)}{} =  \Omega_\subM( (P_{\mathfrak{g}_S}(\eta) )_\subM , X_\nh) +  \langle \mathcal{J},\mathcal{K}_\subW\rangle ( (P_{\mathfrak{g}_S}(\eta) )_\subM , X_\nh ).
\end{gather*}
Since $\langle \mathcal{J},\mathcal{K}_\subW\rangle (X_{\nh},( P_{\mathfrak{g}_S}(\eta))_\subM ) = 0$ then,
by the $G$-invariance of $\Ham_\subM$,
\begin{gather*}
X_\nh (f_\eta) = \Omega_\subM((P_{\mathfrak{g}_S}(\eta) )_\subM,X_\nh) = - d\Ham_\subM ((P_{\mathfrak{g}_S}(\eta) )_\subM )=0.
\end{gather*}
Assertion $(ii)$ is just a consequence of Proposition~\ref{Prop:3Conditions}$(ii)$,
since $P_{\mathfrak{g}_S}(\eta)  \in \Gamma(\g_S)$ and it
is a~generator of the f\/irst integral $f_\eta$.
\end{proof}

In other words, if $W$ satisf\/ies the
 vertical symmetry condition and
 \begin{gather}\label{Eq:JK(Xnh,xi)=0}
 \langle \mathcal{J},\mathcal{K}_\subW\rangle (X_{\nh},( P_{\mathfrak{g}_S}(\eta))_\subM ) = 0,
 \end{gather}
 then, for $\eta\in \mathfrak{g}$,  $P_{\mathfrak{g}_S}(\eta) \in \Gamma(\mathfrak{g}_\subSS)$
 is a horizontal gauge symmetry and $f_\eta$ is  the associated horizontal gauge momentum.

\begin{Remark}
 When $W$ satisf\/ies the vertical symmetry condition, there is a Lie subgroup $G_\subW$ of $G$ integrating the Lie algebra $\mathfrak{g}_\subWW$ and the nonholonomic system is $G_\subW$-Chaplygin, see~\cite{BF2015}.  The reduction of the nonholonomic system (on the Lagrangian side) by the Lie group $G_\subW$ gives the Lagrange d'Alembert Poincar\'e equations on $T(Q/G_\subW)$. However there is a remaining Lie group denoted by $H = G/G_\subW$ that is a symmetry of the equations on $T(Q/G_\subW)$.
 The coordinate version of condition \eqref{Eq:JK(Xnh,xi)=0} appears already in \cite[Theorem~6]{Zenkov} but written on $T(Q/G_\subW)$, after performing the reduction by $G_\subW$.
 \end{Remark}

\begin{Proposition}\label{P:RankS}
 Consider a nonholonomic system $(Q,L,D)$ with a $G$-symmetry that sa\-tis\-fies
 the dimension assumption. If the vertical complement $W$ satisfies the vertical
 symmetry condition and also if $\langle \J,\mathcal{K}_\subW\rangle(X_{\emph\nh}, \eta_\subM) = 0$
 for all $\eta\in \g$, then  there are $k:= \operatorname{rank} S$ linearly independent
 horizontal gauge momenta of the form
 \begin{gather}\label{Eq:f_eta}
   f_\eta = \langle \J^{\emph\nh}, P_{\g_S}(\eta)\rangle,
 \end{gather}
 for $\eta\in \g$.
\end{Proposition}

\begin{proof}
  By the vertical symmetry condition,  let us consider a basis of constant sections
  $\{\tilde \eta_1, \dots$, $\tilde \eta_n\}$ of $\g_\subWW$. Complete now the basis
  so that $\mathfrak{B} = \{\eta_1,\dots, \eta_k , \tilde \eta_1, \dots , \tilde \eta_n\}$
  is a basis of $\g$ (i.e., $\eta_i$ are constant sections of $\g\times Q\to Q$).
  Let us def\/ine $f_i = \langle \J^\nh, P_{\g_S}(\eta_i)\rangle$, $i=1,\ldots, k$.
  We claim that $f_i$, $i=1,\ldots, k$ are $k$ functionally independent f\/irst
  integrals of the nonholonomic dynamics.
  Since $\langle \J,\mathcal{K}_\subW\rangle(X_\nh, P_{\g_S}(\eta_i) \rangle = 0$, $i=1,\ldots, k$,
  the functions $f_1, \ldots, f_k$ are constants along the motions by Proposition~\ref{Prop:Conserved}.
  To see that they are functionally independent, we need to show that the function
  $F= (f_1,\dots , f_k)\colon \M\longrightarrow \R^k$ is a submersion. Now suppose that
  $c_i  d\langle \J^\nh , P_{\g_S}(\eta_i) \rangle = 0$ on $\M$ for $(c_1,\dots , c_k)\in\R^k$ and
  $(c_1,\ldots, c_k)\ne (0,\ldots, 0)$ (or at least that $c_i d\langle \J^\nh , P_{\g_S}(\eta_i)\rangle = 0$
  on an open set $U\subset \M$). In particular, $d\langle \J^\nh , P_{\g_S}(\eta)\rangle|_\subC = 0$ for $\eta= c_i  \eta_i$.
  On the other hand, again using \cite[Theorem~6.5(ii)]{Paula} (as in proof of Proposition~\ref{Prop:Conserved}) we have that ${\bf i}_{( P_{\mathfrak{g}_S}(\eta))_\subM} \Omega_{\mbox{\tiny{${\mathcal J} {\mathcal K}$}}} |_\C= d\langle \J^\nh , P_{\g_S}(\eta)\rangle|_\C = 0 $ for $\Omega_{\mbox{\tiny{${\mathcal J} {\mathcal K}$}}} := \Omega_\subM + \langle \mathcal{J},\mathcal{K}_\subW\rangle$.  Since $ \Omega_{\mbox{\tiny{${\mathcal J} {\mathcal K}$}}} |_\C$ is nondegenerate (\cite[Section~5.1]{Paula} or \cite[Proposition~3]{PL2011}), then $(P_{\g_S}(\eta))_\subM = 0$ which implies that $\eta = c_i  \eta_i\in \g_\subWW$ that is in contradiction with the choice of our basis.
\end{proof}

\section{Examples}\label{S:Examples}

\subsection{Vertical rolling disk}\label{Ex:Disk}

We consider a (homogeneous) disk which is constrained to roll without sliding on a horizontal
plane while standing vertically \cite{BKMM, FSG2008, NF, Pars}.
The conf\/iguration manifold is $Q = \R^2\times S^1\times S^1$ with coordinates $(x,y,\varphi,\psi)$,
where $(x,y)\in \R^2$ are Cartesian coordinates of the contact point,
$\psi$ is the angle between the $x$-axis and the projection of the disk on the plane and $\varphi$
is the angle between a f\/ixed radius of the disk and the vertical. Assuming that the
disk has unit mass the Lagrangian is
\begin{gather*}
  L = \frac{1}{2} \big(\dot x^2+\dot y^2\big) +\frac{1}{2 } I \dot \varphi^2 + \frac{1}{2 } J \dot\psi^2,
\end{gather*}
where $I$ and $J$ are the pertinent moments of inertia.
The nonholonomic constraints describing the rolling motion without slipping are given by
\begin{gather*}
     \dot x = R \cos\psi  \dot\varphi, \qquad
     \dot y = R \sin\psi  \dot\varphi,
 \end{gather*}
and they def\/ine the constraint distribution $D$ on $Q$ with f\/ibers
\begin{gather*}
 D_{(x,y,\varphi, \psi)} = \operatorname{span} \left\{ R \cos\psi  \partial_x + R\sin\psi   \partial_y+ \partial_\varphi, \partial_\psi  \right\} .
\end{gather*}
The group ${\rm SE}(2)\times S^1$ is a symmetry of the system with respect to the action on $Q$ given by
\begin{gather*}
  \big((a,b,\lambda),\mu \big).(x,y,\varphi,\psi) = (a+x\cos \lambda- y\sin\lambda,
  b+ x\sin\lambda+y\cos\lambda,\varphi+\lambda,\psi+\mu)
\end{gather*}
for $(a,b,\lambda)\in {\rm SE}(2)$ and $\mu\in S^1$.

Let us now pass to the Hamiltonian formulation on the cotangent bundle. In  canonical coordinates $(x,y,\varphi, \psi ;p_x,p_y,p_\varphi,p_\psi)$ on $T^*Q$,
the constraint submanifold $\M$ is given by
\begin{gather*}
\M = \left\{(x,y,\varphi,\psi,p_x,p_y,p_\varphi,p_\psi)\colon
p_x = \frac{R}{I} p_\varphi \cos\psi \ \textrm{and} \ p_y = \frac{R}{I} p_\varphi \sin\psi \right\}   ,
\end{gather*}
and since  the Hamiltonian on $T^*Q$ is
$\Ham = \frac{1}{2} (p_x^2+p_y^2) +\frac{1}{2 I} p_\varphi^2 + \frac{1}{2 J} p^2_\psi$, then $\Ham_\subM$ reads
$\Ham_\subM = \frac{1}{2 I} \big( \frac{R^2}{I} +1\big) p_\varphi^2 + \frac{1}{2 J} p^2_\psi$.

Observe that the vertical distribution $V$ has f\/ibers $V_{(x,y,\varphi,\psi)} = \operatorname{span} \{ \partial_x,\partial_y,
\partial_\varphi,\partial_\psi \}$, and that $S = V\cap D = D$.
Therefore the vertical complement $W$ can be chosen to be $W=\operatorname{span}\{ \partial_x, \partial_y\}$ and we can check that it satisf\/ies the vertical symmetry condition, Def\/inition~\ref{D:VertSym}. Moreover $\langle \J, \mathcal{K}_\subW \rangle = 0$ (see \cite{Paula}), since
\begin{gather*}
\J = \left(\frac{R}{I} p_\varphi \cos\psi , \frac{R}{I} p_\varphi \sin\psi,p_\varphi,p_\psi \right)  \in \g^*,
\end{gather*}
and the $\W$-curvature is $\mathcal{K}_\subW =(\sin \psi   d\psi \wedge d\varphi
  , -\cos \psi   d\psi \wedge d\varphi   , 0, 0 ).$
Thus the hypotheses of Propositions~\ref{Prop:Conserved} and~\ref{P:RankS} are satisf\/ied and then the system admits two independent horizontal gauge
momenta.
Since
\begin{gather*}
\g_\subSS = \operatorname{span} \{(R\cos\psi,R\sin\psi,1,0),  (-y,x,0,1)\}, \qquad   \g_\subWW = \operatorname{span} \{(1,0,0,0),  (0,1,0,0)\},
\end{gather*} then
\begin{gather*}
     \xi_1 := P_{\g_S} ((0,0,1,0)) = (R\cos\psi,R\sin\psi,1,0), \qquad
     \xi_2 := P_{\g_S} ((0,0,0,1)) = (-y,x,0,1),
 \end{gather*}
are horizontal gauge symmetries.
In fact, the associated $\M$-cotangent lifts are
$\left(\xi_1\right)_\subQQ^\subM = (\xi_1)_\subM = R\cos\psi \; \partial_x+ R\sin\psi \; \partial_y + \partial_\varphi$
and $\left(\xi_2\right)_\subQQ^\subM = (\xi_2)_\subM =  \partial_\psi$ and thus $\left(\xi_1\right)_\subQQ^\subM (\Ham_\subM) = \left(\xi_2\right)_\subQQ^\subM (\Ham_\subM) = 0$.
Therefore $\J_1 =  \langle \J^{\nh},\xi_1\rangle = \big(\frac{R^2}{I}+1\big)  p_\varphi$
and $\J_2 =  \langle \J^{\nh}, \xi_2\rangle = p_\psi$ are
horizontal gauge momenta.

\subsection{The ball rolling on a plane}\label{Ex:Ball}

Consider the classical example of a inhomogeneous ball of radius $r$ whose center of mass coincides with the geometric center of the ball that rolls without sliding on a plane \cite{BoMa2001,BoMaKi2002}. We will follow the notation and viewpoint of \cite{Paula, Luis}.  The conf\/iguration manifold is $Q = \SO(3) \times \R^2$ where $g \in \SO(3)$ represents the orientation of the ball and $(x,y)\in \R^2$ the point of contact of the ball with the plane. Consider the (left invariant) moving frame $\{X_1^L, X_2^L, X_3^L \}$ on $T_g(\SO(3))$ given by left translations of the canonical basis $\{e_1, e_2, e_3\}$ of $\so (3) \simeq \R^3$ and denote by $\vecOm = (\Omega_1, \Omega_2, \Omega_2)$ the angular velocities of the body with respect to this moving frame.  Then, the non-sliding constraints are given by
\begin{gather*}
\dot x = r \langle \vecbeta, \vecOm\rangle \qquad \mbox{and} \qquad  \dot y = - r\langle \vecalpha, \vecOm\rangle,
\end{gather*}
where $\vecalpha$ and $\vecbeta$ are the f\/irst and second rows of the matrix $g\in\SO(3)$ and $\langle \cdot, \cdot \rangle$ is the canonical pairing in $\R^3$. The non-integrable distribution $D$ is then generated by
\begin{gather*}
\left\{ X_1 := X_1^L + r\beta_1 \frac{\partial}{\partial x} - r \alpha_1 \frac{\partial}{\partial y} , \; X_2 := X_2^L + r \beta_2 \frac{\partial}{\partial x} - r\alpha_2 \frac{\partial}{\partial y},\right.\\
\left. X_3 : = X_3^L + r \beta_3 \frac{\partial}{\partial x} - r\alpha_3 \frac{\partial}{\partial y} \right\} .
\end{gather*}
Let us consider the action of the Lie group $G=\{ (h,a) \in \SO(3) \times \R^2 \colon h\cdot e_3 = e_3 \}$ on $Q$ given by
\begin{gather*}
(h,a) \colon \ (g, (x,y)) \mapsto (h g, (x,y) \tilde h^t + a) \in Q,
\end{gather*}
where $\tilde h$ is the $2\times 2$ rotational matrix determined by $h$.  This action is a symmetry of the nonholonomic system since it leaves invariant the distribution $D$ and the lagrangian $L (g, x,y; \vecOm, \dot x, \dot y) = \frac{1}{2}\left( \langle \mathbb{I} \vecOm, \vecOm \rangle + m( \dot x^2 + \dot y^2) \right),
$ where $\mathbb{I}$ the inertia tensor ($\mathbb{I}$ is represented by a diagonal matrix with entries $\mathbb{I}_1$, $\mathbb{I}_2$, $\mathbb{I}_3$). Note that the Lie algebra $\g$ associated to $G$ is $\R\times \R^2$ with the trivial structure and also
\begin{gather*}
(1;0,0)_\subQQ = \langle \vecgamma, {\bf X} \rangle -y\frac{\partial}{\partial x} + x\frac{\partial}{\partial y}, \qquad (0;1,0)_\subQQ = \frac{\partial}{\partial x} \qquad \mbox{and} \qquad (0;0,1)_\subQQ = \frac{\partial}{\partial y},
\end{gather*}
where $\vecgamma = (\gamma_1, \gamma_2, \gamma_3)$ is the third row of the rotational matrix $g$ and ${\bf X} = (X_1,X_2, X_3)$.
Therefore, we can choose the vertical complement $W$ to be $W=\operatorname{span}\{Z_1:=\partial_x, Z_2:= \partial_y\}$ and then
\begin{gather*}
\g_\subSS = \operatorname{span}\{(1;y,-x)\} \qquad \mbox{and} \qquad \g_\subWW =\operatorname{span}\{(0;0,1), (0;1,0) \}.
\end{gather*}

Now we will def\/ine the submanifold $\M\subset T^*Q$. Let us denote by $\vecL = (\lambda_1, \lambda_2, \lambda_3)$ where $\{\lambda_1, \lambda_2, \lambda_3\}$ is the (local) basis of $T^*(\SO(3))$ that is dual to the moving frame $\{X_1^L, X_2^L, X_3^L\}$.  Following \cite{Paula, Luis} we consider the basis $\{\lambda_1, \lambda_2, \lambda_3, \epsilon_x := dx - r \langle \vecbeta,\vecL \rangle , \epsilon_y := dy + r \langle \vecalpha,\vecL \rangle \}$ of $T^*Q$ that is dual to the adapted basis $\{X_1, X_2, X_3 ,Z_1, Z_2\}$ of $TQ = D\oplus W$. Then if $(K_1, K_2, K_3, \tilde p_x, \tilde p_y)$ are the coordinates on $T^*Q$ associated to the dual basis, the submanifold~$\M$ is given by
\begin{gather*}
\M = \big\{ (g,x,y; K_1, K_2, K_3, \tilde p_x, \tilde p_y) \colon {\bf K} = \mathbb{I}\vecOm + mr^2 (\vecOm - \langle \vecgamma, \vecOm \rangle \vecgamma ), \\
\hphantom{\M = \big\{ (g,x,y; K_1, K_2, K_3, \tilde p_x, \tilde p_y) \colon}{} \tilde p_x = mr \langle \vecbeta,\vecOm\rangle \  \mbox{and} \  \tilde p_y = - mr \langle \vecalpha, \vecOm\rangle  \big\} ,
\end{gather*}
where ${\bf K} = ( K_1, K_2, K_3)$.

An arbitrary section $\xi\in \Gamma(\g_\subSS)$ has the form
\begin{gather*}
\xi = \phi(g,(x,y)) (1; y,-x)= \phi.(1;0,0) + y \phi. (0;1,0) - y \phi.(0;0,1),
\end{gather*}
for $\phi\in C^\infty(Q)$ and also we where we denote $\phi = \phi(g,(x,y))$.
Then from Lemma~\ref{L:C-lift} the $\M$-cotangent lift of $\xi$  is
\begin{gather*}
  \xi_\subQQ^\subM =     \phi \langle \vecgamma, {\bf X} \rangle  + \left(  \langle \vecgamma, {\bf K} \rangle -  mr(y \langle \vecbeta,\vecOm\rangle +  x\langle \vecalpha,\vecOm\rangle)  \right) \pi_\nh^\sharp (d\phi) \\
\hphantom{\xi_\subQQ^\subM = }{} + mr \big( \langle \vecbeta,\vecOm\rangle \pi_\nh^\sharp(d(y\phi)) +  \langle \vecalpha,\vecOm\rangle \pi_\nh^\sharp(d(x\phi)) \big),
\end{gather*}
where we used that $\Theta_\subM = \langle {\bf K}, \vecL \rangle + \iota^* (\tilde p_x) \epsilon_x + \iota^*(\tilde p_y) \epsilon_y$ to compute $\langle \J, \eta_i\rangle$.
Knowing that the restricted hamiltonian $\Ham_\subM\colon \M\to \R$ is $\Ham_\subM = \frac{1}{2} \langle {\bf K},  \vecOm \rangle$, following Proposition~\ref{Prop:3Conditions}(ii) we should look for a function $\phi\in C^\infty(Q)$ such that
\begin{gather*}
d\Ham_\subM \big((\langle \vecgamma, {\bf K} \rangle -  mr\langle y\vecbeta + x\vecalpha ,\vecOm\rangle   ) \pi_\nh^\sharp (d\phi)\\
 \qquad{} + mr ( \langle \vecbeta,\vecOm\rangle \pi_\nh^\sharp(d(y\phi)) +  \langle \vecalpha,\vecOm\rangle \pi_\nh^\sharp(d(x\phi)) ) \big) = 0
\end{gather*}
to guarantee that $\xi$ is a horizontal gauge symmetry and so $\J_\xi = \langle \J^\nh, \xi\rangle$ is a f\/irst integral.  However, we can check that the vertical complement $W$ satisf\/ies the vertical symmetry condition (see Def\/inition~\ref{D:VertSym}) and therefore we have a natural candidate for the horizontal gauge symmetry: from Proposition~\ref{Prop:Conserved}, $\xi = P_{\g_S}(1;0,0) = (1; y, -x)$ is a horizontal gauge symmetry if $\langle \mathcal{J},\mathcal{K}_\subW\rangle (X_{\nh}, \xi_\subM )=0$ where, in this case, $\xi_\subM = \langle \vecgamma, {\bf X}\rangle$.  From \cite{Paula} we have that
\begin{gather*}
\langle \J,\mathcal{K}_\subW \rangle = mr^2 \langle \vecOm - \langle \vecgamma,\vecOm \rangle\vecgamma , d\vecL \rangle,
\end{gather*}
where $d\vecL = (\lambda_2\wedge\lambda_3, \lambda_3\wedge\lambda_1, \lambda_1\wedge\lambda_2)$. It is easy to check that $\langle \mathcal{J},\mathcal{K}_\subW\rangle (X_{\nh},\langle \vecgamma, {\bf X} \rangle )=0$ for $X_\nh = \langle \vecOm, {\bf X} \rangle + \langle {\bf K}\times \vecOm , \partial_{\bf K} \rangle$. Therefore the section $\xi = P_{\g_S}(1;0,0) = (1; y, -x)$ is the horizontal gauge symmetry and $\J_\xi = {\bf i}_{\langle \vecgamma, {\bf X} \rangle} \Theta_\subM = \langle \vecgamma, {\bf K} \rangle$ is a f\/irst integral of the dynamics, i.e., $\J_\xi$ is a~horizontal gauge symmetry.

From Proposition~\ref{P:RankS} we observe that $\operatorname{rank}S =1$ and so $\J_\xi$ is the only (up to constants) horizontal gauge symmetries of the form \eqref{Eq:f_eta}.

\begin{Remark} Observe that $\pounds_{\xi_\subM} \Omega_\subM = d\J_\xi - mr\langle \vecOm, d\vecgamma\rangle + \Sigma$ where $\Sigma$ satisf\/ies that $\Sigma (X) = 0$ for all $X\in \Gamma(\C)$. However, it is easy to check that $mr\langle \vecOm, d\vecgamma\rangle |_\C \neq 0 $ and therefore, the vector f\/ield $\xi_\subM = \langle \vecgamma, {\bf X} \rangle$ is not a {\it Cartan symmetry} in the sense of~\cite{Crampin-Mestdag}, see Remark~\ref{R:Cartan}.
\end{Remark}

\section{Conclusions and perspectives}
The main goal of this work is to show how certain f\/irst integrals linear in the momenta of a~nonholonomic systems with a $G$-symmetry might be generated by the symmetry group.
In particular, in Def\/inition~\ref{Def:M-lift} and Corollary~\ref{C:MLift} we def\/ine the $\M$-cotangent lift and we use it to characterize
horizontal gauge symmetries.
Then, under the dimension assumption~\eqref{Eq:DimAssumptionTQ}, the tangent
bundle can be splitted (see~\eqref{Eq:VertComplementTQ})
in the direct sum of the constraint distribution~$D$ and a~vertical complement~$W$.
If~$W$ satisf\/ies the so-called vertical symmetry condition (Def\/inition~\ref{D:VertSym}),
Proposition~\ref{Prop:Conserved} guarantees  that the nonholonomic system admits
horizontal gauge momenta induced by elements of the Lie algebra, while
Proposition~\ref{P:RankS} tells that the number of (linearly independent)
horizontal gauge momenta is given by the rank of the distribution $S=D\cap V$.

It is important to note that the choice of the vertical complement~$W$ is not unique.
Usually, the most natural choice for $W$ is a complement that satisf\/ies the vertical
symmetry condition, as we did in Examples~\ref{Ex:Disk} and~\ref{Ex:Ball}. This condition
is a `natural' condition that arises in the framework of reduction and Hamiltonization,
see~\cite{Paula}.  However, in some cases choosing a complement that satisf\/ies
the vertical symmetry condition might be too strong to study a group origin
of f\/irst integrals (linear in the momenta).  The cases where we need a vertical
complement~$W$ that does not satisfy the vertical symmetry condition is our
topic of future work and investigation.

More precisely, in the following example we show that when we choose a
complement $W$ satisfying the vertical symmetry
condition then, for $\eta\in \g$,   $\langle \mathcal{J},\mathcal{K}_\subW\rangle (X_\nh,( P_{\mathfrak{g}_\subSS}(\eta))_\M ) \neq 0$  (cf. Proposition~\ref{Prop:Conserved}).
However, the choice of a dif\/ferent vertical complement $W$ allows us to write the horizontal gauge symmetry in terms of a section $P_{\g_S}(\eta)\in \Gamma(\g_\subSS)$ for $\eta\in\g$; but in this case the vertical complement $W$ does not verify the vertical symmetry condition.

 \begin{Example}[the nonholonomic particle~\cite{BGMConservation, Pars}]
 The classical example of the nonholonomic particle describes a particle in $\R^3$
 subjected to the nonholonomic constraints $\dot z = y \dot x$ and where the
 Lagrangian is the canonical kinetic energy metric on $\R^3$. Then, in canonical coordinates
 the distribution $D$ has f\/ibers $D_{(x,y,z)} = \operatorname{span}\{ \partial_y, \partial_x+y\partial z \}$,
 the constraint manifold is
 \begin{gather*}
 \M = \{(x,y,z; p_x ,p_y , p_z)\colon  p_z = yp_x \},
 \end{gather*}
 and the restricted Hamiltonian is $\Ham_\subM = \frac{1}{2} ( (1+y^2)p_x^2+p_y^2 )$.
 The nonholonomic vector f\/ield is
 \begin{gather*}
 X_\nh = p_x\left(\frac{\partial}{\partial x} + y\frac{\partial}{\partial z} \right) + p_y\frac{\partial}{\partial y} - \frac{yp_xp_y}{1+y^2} \frac{\partial}{\partial p_x}.
 \end{gather*}

It is well known that the function $f(x,y,z,p_x,p_y) = p_x   \sqrt{1+y^2}$
is a f\/irst integral of the system and that it is a horizontal gauge momentum,
with respect to the $G$-symmetry given by the translational action of $\R^2$ on $\R^3$.
The vertical space associated to the $G$-action is $V = \operatorname{span}\{ \partial_x, \partial_z \}$, and
the subbundle $\g_\subSS$ is locally generated by $\{(1,y)\}$.
Then the function $f$ is a~horizontal gauge momentum with horizontal gauge symmetry
 $\xi =\sqrt{1+y^2}  (1,y) \in \Gamma(\mathfrak{g}_\subSS)$, i.e., $\xi_\subQQ^{T^*Q}(\Ham)(m)=0$ for $m\in\M$ and thus $f=\J_\xi$ (see~\cite{FSG2008}).

Now we are going to give a geometric interpretation of $\xi$ in terms of the tools of Section~\ref{S:GaugeSym},
that is, by choosing a vertical complement of the constraint distribution.
An obvious choice for the $G$-invariant vertical complement
$W$ so that it satisf\/ies the vertical symmetry condition is
$W = \operatorname{span}\{ \partial_z \}$.\footnote{This is the choice
done in \cite{Paula} in the framework of Hamiltonization.}
Observe that $\mathfrak{g}_\subWW = \operatorname{span}\{(0,1)\}$.
It is straightforward to check that for any $(a,b) \in \mathfrak{g} = \R^2$,
 $P_{\mathfrak{g}_\subSS} ((a,b) ) = a (1,y)$, so there is no element
 $\eta\in\mathfrak{g}$ such that $P_{\mathfrak{g}_\subSS} (\eta )$ gives the
horizontal gauge symmetry $\xi= \sqrt{1+y^2}(1,y)$.
Indeed, using that $\langle \mathcal{J},\mathcal{K}_\subW\rangle = yp_x dx\wedge dy$, we check that
$\langle \mathcal{J},\mathcal{K}_\subW\rangle (X_{\nh},( P_{\mathfrak{g}_\subSS}((1,0)))_\subM )\neq 0$ (see~\cite{Paula}).

However, it is possible to change $W$ in order to obtain a horizontal gauge symmetry
and a~horizontal gauge momentum of the form
$P_{\mathfrak{g}_\subSS}(\eta) \in \Gamma(\mathfrak{g}_\subSS)$
and $f_\eta= \langle \J^\nh, P_{\mathfrak{g}_\subSS}(\eta) \rangle$.

Consider the $G$-invariant vertical complement $W$ whose f\/ibers are given by
\begin{gather}\label{eq:W}
 W_{(x,y,z)}= \operatorname{span} \left\{ \big(1-\sqrt{1+y^ 2}\big) \frac{\partial}{\partial x} + y \frac{\partial}{\partial z} \right\}.
\end{gather}
The splitting  \eqref{Eq:gS+gW} of the Lie algebra $\g$ is given by
\begin{gather} \label{Ex:g-splitting}
 \mathfrak{g}_\subSS = \operatorname{span} \{ (1,y)\} \qquad \mbox{and} \qquad \mathfrak{g}_\subWW = \operatorname{span}\big\{\big(1-\sqrt{1+y^ 2} ,y\big)\big\}  .
\end{gather}
Observe that $W$ does not satisfy the vertical symmetry condition, however
\begin{gather*}
  P_{\mathfrak{g}_\subSS} ((1,0)) = \sqrt{1+y^2}(1,y) = \xi,
\end{gather*}
where $P_{\mathfrak{g}_\subSS} \colon \R^2 \to \mathfrak{g}_\subSS$ is the projection
associated to the splitting def\/ined in~\eqref{Ex:g-splitting}.
Observe that the projection $P_{\mathfrak{g}_\subSS}$ is induced by the choice of the vertical complement~$W$ given in~\eqref{eq:W}.
In this case, the horizontal gauge momentum~$f$ is exactly
$f = \langle \J^\nh, P_{\mathfrak{g}_\subSS}((1,0)) \rangle.$
\end{Example}

It is then clear that a deeper study of this topic is necessary as a more
geometric understanding of the Noetherian character (see~\cite{FSG2008, FSG2009})
of horizontal gauge momenta (see~\cite{BS2016}).

\subsection*{Acknowledgments}

This work is partially supported by the research projects
{\it Symmetries and integrability of nonholonomic mechanical systems} of the University of Padova.
N.S.~wishes to thank IMPA and H.~Bursztyn for the kind hospitality during which this work took origin.  P.B.~thanks the f\/inancial support of CAPES (grants PVE 11/2012 and PVE 089/2013) and CNPq's Universal grant. We also thank the anonymous referees for their useful comments.

\pdfbookmark[1]{References}{ref}
\LastPageEnding

\end{document}